\documentclass[final,reqno]{amsart}
\usepackage{soul,cancel, ulem}
\usepackage[color]{showkeys}
\usepackage{graphicx}
\definecolor{refkey}{rgb}{0,1,0}
\definecolor{labeled}{rgb}{1,0,0}

\usepackage{amsmath}
\usepackage{amsfonts}
\usepackage{amssymb}
\usepackage{color}
\usepackage{srcltx}
\usepackage{fullpage}
\usepackage{hyperref}
\usepackage{mathrsfs}
\usepackage{subcaption}

\usepackage{wrapfig}

 \newtheorem{thm}{Theorem}
 
 \newtheorem{lem}[thm]{Lemma}
 \newtheorem{prop}[thm]{Proposition}
 \theoremstyle{definition}
 
 \theoremstyle{remark}
 
 \newtheorem{ex}[thm]{Example}

\newcommand{\eq} [1] {\begin{equation}\label{#1}\quad}
\newcommand{\en} {\end{equation}}
\newcommand{\scal}[1]{\langle#1\rangle}
\newcommand{\norm}[1]{\left\Vert#1\right\Vert}
\newcommand{\abs}[1]{\left\vert#1\right\vert}

\newcommand{\re}{\operatorname{Re}}
\newcommand{\im}{\operatorname{Im}}
\newcommand{\Span}{\operatorname{Span}}

\newcommand{\C}{\mathbb{C}}

\newcommand{\R}{\mathbb{R}}

\newcommand{\diag}{\operatorname{diag}}

\usepackage{soul}

\renewcommand{\em}{\it}

\begin{document}
%
%

\title{The Gau--Wu Number for $4\times4$ and Select Arrowhead Matrices\\ \today}

\author[Camenga]{Kristin A. Camenga}
\address{Department of Mathematics\\
Juniata College\\
Huntingdon, PA\\ USA}
\email{camenga@juniata.edu}

\author[Rault]{Patrick X. Rault}
\address{Department of Mathematics\\ University of Nebraska at Omaha\\
Omaha, NE 68182\\
USA}
\email{prault@unomaha.edu}

\author[Spitkovsky]{Ilya M. Spitkovsky}
\address{Division of Science \\ New York University Abu Dhabi (NYUAD)\\ Abu Dhabi 129188\\
United Arab Emirates}
\email{ims2@nyu.edu, imspitkovsky@gmail.com}

\author[Yates]{Rebekah B. Johnson Yates}
\address{Department of Mathematics\\Houghton College\\Houghton, NY 14744\\USA}
\email{rebekah.yates@houghton.edu}

\thanks{The work on this paper was prompted by the authors' discussions during a SQuaRE workshop in May 2013 and continued during a REUF continuation workshop in August 2017, both at the American Institute of Mathematics and supported by the NSF (DMS \#1620073).  The third author [IMS] was also supported by Faculty Research funding from the Division of Science and Mathematics, New York University Abu Dhabi.}

\begin{abstract}
The notion of dichotomous matrices is introduced as a natural generalization of essentially Hermitian matrices. A criterion for arrowhead matrices to be dichotomous is established, along with necessary and sufficient conditions for such matrices to be unitarily irreducible. The Gau--Wu number (i.e., the maximal number $k(A)$ of orthonormal vectors $x_j$ such that the scalar products $\scal{Ax_j,x_j}$ lie on the boundary of the numerical range of $A$) is computed for a class of arrowhead matrices $A$ of arbitrary size, including dichotomous ones. These results are then used to completely classify all $4\times4$ matrices according to the values of their Gau--Wu numbers.
\end{abstract}
\subjclass{Primary 15A60}

\keywords{Numerical range, field of values, $4\times 4$ matrices, Gau--Wu number, boundary generating curve, irreducible, arrowhead matrix, singularity}

\dedicatory{}

\maketitle

\section{Introduction}
The \textit{numerical range} $W(A)$ of an $n\times n$ matrix $A$ is by definition the image of the unit sphere of $\C^n$ under the mapping $f_A\colon x\mapsto \scal{Ax,x}$, where $\scal{.,.}$ is the standard scalar product
on the space $\C^n$.

The set $W(A)$ is compact and convex and has a piecewise smooth boundary $\partial W(A)$. The maximal number $k(A)$
of orthonormal vectors $x_j$ such that $z_j:=f_A(x_j)\in\partial W(A)$ was introduced in \cite{GauWu13} and called
the \textit{Gau--Wu number} of $A$ in \cite{CaRaSeS}.  Clearly, $k(A)\leq n$, and it is not hard to see (\cite[Lemma 4.1]{GauWu13}) that $k(A)\geq 2$. Thus $k(A)=2$ for all $2\times2$ matrices.

A description of matrices for which the bound $k(A)=n$ is attained was given in \cite[Theorem 2.7]{WangWu13}; see also
\cite[Section 4]{CaRaSeS}. In Section~\ref{s:arrow} we derive a more constructive version of this description for the case of arrowhead matrices and also provide conditions sufficient for $k(A)$ to attain its smallest possible value 2. Some preliminary results about these matrices, along with the definition, are provided in Section~\ref{s:irred}. Building on the result for $3\times 3$ matrices from \cite[Proposition 2.11]{WangWu13}, in Section~\ref{s:4x4} we complete the classification of Gau--Wu numbers of arbitrary $4\times4$ matrices, making a step forward.

\section{Unitary (ir)reducibility of arrowhead matrices} \label{s:irred}

An \textit{arrowhead matrix} is a matrix that has all its non-zero entries on the main diagonal, in the last row, and in the last column. In other words, it is a matrix of the form
\eq{ah}
A=\begin{bmatrix}a_1 & 0 & \ldots & 0 & b_1 \\
	0 & a_2 & \ldots & 0 & b_2 \\
	\vdots & \vdots & \ddots & \vdots & \vdots \\
	0 & 0 & \ldots & a_{n-1} & b_{n-1} \\
	c_1 & c_2 & \ldots & c_{n-1} & a_n \end{bmatrix}.
\en
When $n\leq 2$, every matrix is an arrowhead matrix.  We will therefore focus on the situation when $n\geq 3$ in this paper.

Sometimes it will be convenient for us to write \eqref{ah} as
\eq{bah} A=\begin{bmatrix} A_0  & b \\ c & \mu\end{bmatrix},\en
where $A_0 =\diag[a_1,\ldots,a_{n-1}]$, $\mu=a_n$, $b$ is the column $[b_1,b_2,\ldots, b_{n-1}]^T$, and $c$ is the row $[c_1,c_2,\ldots, c_{n-1}]$.

Our first observation is trivial, and we put it in writing (without proof) for convenience of reference only.
\begin{lem} \label{l:eig}Suppose $\lambda$ is an eigenvalue of \eqref{ah} different from its diagonal entries $a_1,\ldots,a_{n-1}$. Then \eq{inter} \sum_{j=1}^{n-1}\frac{b_j c_j}{\lambda-a_j}+a_n-\lambda=0,\en the respective eigenspace is one-dimensional, and it is spanned by the vector $x$ with the
coordinates \eq{eig} x_j=\frac{b_j}{\lambda-a_j}, \ j=1,\ldots, n-1; \quad x_n=1. \en \end{lem}

Recall that $\lambda$ is a {\em normal eigenvalue} of a given matrix $A$ if there is a corresponding eigenvector $x$
of $A$ which at the same time is an eigenvector of $A^*$; $x$ is then called a {\em reducing eigenvector} of $A$, and it corresponds to the eigenvalue $\overline{\lambda}$ of $A^*$. Of course, a matrix having a normal eigenvalue is unitarily reducible; the converse is true for $n=2$ or 3 but not in higher dimensions.

Arrowhead matrices having at least one normal eigenvalue are characterized by the following statement.

\begin{prop}\label{p:redsuf} The matrix \eqref{ah} has a normal eigenvalue if and only if at least one of the following (not mutually exclusive) conditions holds:\\
{\rm (i)} $b_j=c_j=0$ for at least one  $j\in\{1,\ldots,n-1\}$; \\
{\rm (ii)} $a_i=a_j$ and $b_i\overline{c_j}=b_j\overline{c_i}$ for some $i,j\in\{1,\ldots,n-1\}$, $i\neq j$;\\
{\rm (iii)} $a_i=a_j=a_k$ for some triple of distinct $i,j,k\in\{1,\ldots.n-1\}$; \\
{\rm (iv)} $\abs{b_j}=\abs{c_j}\neq 0$ for all $j=1,\ldots,n-1$, and the lines $\ell_j$ passing through $a_j$ and forming the angle $(\arg{b_j}+\arg{c_j})/2$  with the positive real half-axis either all coincide or all intersect at (then defined uniquely) $\lambda$ satisfying \eqref{inter}.
\end{prop}
\begin{proof} First we will prove that $A$ has a normal eigenvalue $\lambda$ associated with a respective eigenvector $x$, the last coordinate of which is equal to zero, if and only if (i), (ii), or (iii) holds.
	
Indeed, for such $x$,  conditions \eq{ne} Ax=\lambda x \text{ and } A^*x=\overline{\lambda}x \en can be rewritten as the homogeneous system of linear equations
	\eq{sys1} \begin{aligned} (a_j-\lambda) x_j & =  0, &  j=1,\ldots, n-1, \\
	\sum_{j=1}^{n-1}c_j x_j & =  0,  &
	\sum_{j=1}^{n-1}\overline{b}_j x_j  =  0.  \end{aligned} \en

For $\lambda$ different from all $a_1,\ldots,a_{n-1}$, the first $n-1$ equations imply that $x=0$, and so these $\lambda$'s cannot be normal eigenvalues.

If $\lambda=a_j$ for exactly one $j$,  we arrive at a non-trivial solution if and only if the respective $b_j,c_j$ are equal to zero. This is case (i).

If $\lambda=a_i=a_j\neq a_k$ for some $i\neq j$ and all other $k$, with $i,j,k\in\{1,\ldots,n-1\}$, then the first line of \eqref{sys1} is equivalent to $x_k=0$ if $k\neq i,j$, and a non-trivial solution of \eqref{sys1} exists if and only if this is true for  \[
c_ix_i+c_jx_j=0, \text{ and } \overline{b_i}x_i+\overline{b_j}x_j=0, \]
i.e., if and only if we are in case (ii).

Finally, if $a_i=a_j=a_k$ with distinct $i,j,k$---that is, we are in case (iii)---we have that a non-trivial solution $\{x_i,x_j,x_k\}$ to
\[
c_ix_i+c_jx_j+c_kx_k=0, \text{ and } \overline{b_i}x_i+\overline{b_j}x_j+\overline{b_k}x_k=0, \]
with all other coordinates of $x$ set to zero, delivers a non-trivial solution to \eqref{sys1}.

In what follows, we may suppose that (i) does not hold, i.e., for every $j=1,\ldots,n-1$ at least one of the off diagonal entries $b_j,c_j$ is different from zero. Equations \eqref{ne} then imply that
\[ a_jx_j+b_jx_n=\lambda x_j, \quad \overline{a_j}x_j+\overline{c_j}x_n=\overline{\lambda}x_j,\quad j=1,\ldots, n-1.\]	
The reducing eigenvector $x$ with $x_n\neq 0$ may therefore correspond only to an eigenvalue different from $a_1,\ldots,a_{n-1}$. Applying Lemma~\ref{l:eig} (also with $A$ replaced by $A^*$) 	we reach the conclusion that \eqref{ne} holds if and only if \eqref{inter} is satisfied, and
\[ b_j/(\lambda-a_j)=\overline{c_j/(\lambda-a_j)}, \quad j=1,\ldots,n-1.\]
Equivalently, \[ \abs{b_j}=\abs{c_j} \text{ and } \arg(\lambda-a_j)=(\arg b_j+\arg c_j)/2\mod \pi. \]
This is in turn equivalent to (iv). 	
\end{proof}

From the proof of Proposition~\ref{p:redsuf} it is clear that ${\rm (i)}\vee{\rm (ii)}\vee{\rm (iii)}$ holds if and only if there exists a joint eigenvector of $A$ and $A^*$ orthogonal to $e_n$. Interestingly, the absence of such eigenvectors implies certain restrictions on reducing subspaces of $A$.

\begin{lem}\label{l:red}
Let an $n\times n$ arrowhead matrix $A$ be such that its reducing eigenvectors (if any) are not orthogonal to $e_n$. Then a reducing subspace of $A$ that contains at least one of the standard basis vectors $e_j$ coincides with the whole space $\C^n$.
\end{lem}

\begin{proof}
Let $L$ be a subspace invariant both under $A$ given by \eqref{ah} and $A^*$. If $L$ contains $e_j$ for some $j<n$, then also $b_je_n\in L$ and $\overline{c_j}e_n\in L$. Since the negation of (i) from  Proposition~\ref{p:redsuf} holds, we conclude that $e_n\in L$. Thus, this case is reduced to $e_n\in L$.

Using notation \eqref{bah}, observe that from $e_n\in L$ it follows that $Ae_n-\mu e_n=\begin{bmatrix}b\\ 0\end{bmatrix}$
and $A^*e_n-\overline{\mu}e_n=\begin{bmatrix}c^*\\ 0\end{bmatrix}$ also lie in $L$. Moreover, along with every $x=\begin{bmatrix}x'\\x_n\end{bmatrix}\in L$ we also have $\begin{bmatrix}A_0  x'\\ 0\end{bmatrix}\in L$.
Using this property repeatedly, we see that
\[ \begin{bmatrix}A_0^m b\\ 0\end{bmatrix}, \begin{bmatrix}A_0^m c^* \\ 0\end{bmatrix}\in L \text{ for all } m=0,1,\ldots \]
The negation of {${\rm (i)}\vee{\rm (ii)}\vee{\rm (iii)}$} means exactly that the system of vectors $\{ A_0^m b, A_0^m c^*\colon m=0,1,\ldots \}$ has the full rank $n-1$, thus implying that $L$ contains all
the vectors $e_j$ and not just $e_n$.
In other words, $L$ is the whole space.
\end{proof}

We note that if in the setting of Proposition~\ref{p:redsuf} case (iv) all the lines $\ell_j$ coincide,  condition \eqref{inter} then implies that $\arg(\lambda-a_n)$ is the same as all $\arg(\lambda-a_j)$, $j=1,\ldots,n-1$. According to \cite[Proposition 2]{JSST}, this is exactly the situation when $A$ is normal and, moreover, {\em essentially Hermitian}; the latter means that $A$ is a linear combination of a Hermitian matrix and the identity.

Another way to look at essentially Hermitian matrices is as follows: given $A\in\C^{n\times n}$, consider the minimal number of distinct eigenvalues attained by the Hermitian (equivalently: skew-Hermitian) part of the ``rotated'' matrix $e^{i\theta}A$. A matrix is essentially Hermitian if and only if this number is the minimal possible, i.e. equals one.

From this point of view it is very natural to introduce the class of matrices for which this number equals two. Let us call such matrices {\em dichotomous}.
We want to describe all dichotomous arrowhead matrices. Let us start with the simplest case when the matrix itself is Hermitian, with 0 and 1 as its eigenvalues. In other words, $A=P$ is a self-adjoint involution
(in operator terms, an orthogonal projection).

\begin{lem}\label{l:op}
An arrowhead matrix $P$ is such that $P=P^*=P^2$ if and only if it is either {\rm (i)} diagonal, with all the diagonal entries equal 0 or 1, or {\rm (ii)} has the form
\eq{opah}
P= \begin{bmatrix}p_1 & 0 & \ldots  & \ldots & \ldots & \ldots  & \ldots & 0 \\
0 & \ddots &  &  &  & &  & \vdots \\
\vdots & & p_{i-1} & & & & & 0 \\
\vdots & & & t & & & & \alpha \\
\vdots & & & & p_{i+1} & & & 0 \\
\vdots & & & & & \ddots & & \vdots \\
\vdots & & & & & & p_{n-1} & 0 \\
0 &\dots  & 0 & \overline{\alpha} & 0 & \ldots & 0 & 1-t\end{bmatrix}.
\en
Here $t\in (0,1)$, $\abs{\alpha}=\sqrt{t(1-t)}$, $\alpha$ and $\overline{\alpha}$ are the $i$th elements of the last column and row, respectively, and $p_j\in\{0, 1\}$ for all $j=1,\ldots n-1$ with $j\neq i$.
\end{lem}

\begin{proof}
{\sl Sufficiency} can be checked by a simple direct computation. To prove {\sl necessity}, observe first of all that if an arrowhead matrix $P$ is self-adjoint, then its diagonal entries are real while the block representation takes the form
	\[ P=\begin{bmatrix} P_0  & \omega \\ \omega^*& \mu \end{bmatrix}. \]
Condition $P^2=P$ implies in particular that \eq{ome} \omega\omega^*=P_0 -P_0^2. \en
Since the right hand side of \eqref{ome} is diagonal along with $P_0 $, the entries of $\omega$ are such that $\omega_j\overline{\omega_k}=0$ whenever $j\neq k$. Thus, at most one of them can be different from zero. In the case that $\omega =0$, the matrix $P$ is diagonal, and so (i) holds.

Suppose therefore that for some $i$ we have $\omega_i:=\alpha\neq 0$. Then, up to a permutational similarity, $P$ is the direct sum of the diagonal
matrix $\diag[p_1,\ldots,p_{i-1},p_{i+1},\ldots,p_{n-1}]$ and the $2\times2$ block $P'=\begin{bmatrix}p_i & \alpha \\ \overline{\alpha} & \mu\end{bmatrix}$. Since $P'$ is an orthogonal projection with a non-zero off diagonal entry, $P'$ has rank one and is unitarily similar to $\diag[0,1]$. Denoting $p_i:=t$, we thus see that $\mu=1-t$ and $\abs{\alpha}^2=p_i\mu=t(1-t)$, as described in case (ii).
\end{proof}

The description of all dichotomous arrowhead matrices follows.

\begin{prop}\label{p:dich}
The arrowhead matrix \eqref{ah} is dichotomous if and only if for some angle $\theta$ and distinct real $h_0, h_1$ either	
\begin{enumerate}
\item[(a)]
\begin{equation}\label{offdiag}
\arg{b_j}+\arg{c_j}=:2\theta+\pi \text{ does not depend on } j,\quad 	\abs{b_j}=\abs{c_j}
\end{equation}
\noindent for $j=1,\ldots, n-1$, and
\begin{equation}\label{diag}
	\re(e^{-i\theta}a_j) \text{ equals } h_0 \text{ or } h_1
\end{equation}
for $j=1,\ldots,n$, with both values attained,\end{enumerate}
or \begin{enumerate}
\item[(b)] {Conditions \eqref{diag} and \eqref{offdiag} hold for $j=1,\ldots,n-1$, except for one missing $j=i$, and}
\begin{equation}\label{coni}
	\re(e^{-i\theta}a_i)=(1-t)h_0+th_1, \quad
	\re(e^{-i\theta}a_n)=th_0+(1-t)h_1,
\end{equation}
\begin{equation*}
	\abs{e^{-i\theta}b_i+e^{i\theta}\overline{c_i}}=2\abs{h_1-h_0}\sqrt{t(1-t)}
\end{equation*} for some $t\in (0,1)$.
\end{enumerate}
\end{prop}
By convention, the argument condition in \eqref{offdiag} is considered satisfied whenever $b_j=c_j=0$. Condition \eqref{diag} means that the respective elements $a_j$ lie on two parallel lines
forming the angle $\theta$ with the positive direction of the $y$-axis. The slope of these lines is defined uniquely by \eqref{offdiag} if there is at least one non-zero pair $b_j,c_j$.

\begin{proof}
Any matrix $A$ is dichotomous if and only if for some $\theta$ the self-adjoint matrix $\re(e^{-i\theta}A)$ has exactly two distinct eigenvalues, say $h_0$ and $h_1$.
This is equivalent to
	\eq{AP} \re(e^{-i\theta}A)=h_0 I +(h_1-h_0)P,\en
where $P\neq 0,I$ is a self-adjoint idempotent. If this $A$ is an arrowhead matrix, then so is $P$. Invoking Lemma~\ref{l:op} completes the proof.
\end{proof}

Note that case (a) corresponds to diagonal $P$, while (b) stems from $P$ given by \eqref{opah}. In the latter case $\arg\alpha$ is defined uniquely as $\arg(e^{-i\theta}b_i+e^{i\theta}\overline{c_i})$.

For our purposes, the case of unitarily irreducible dichotomous matrices is of special interest.

\begin{thm}\label{th:diur}
An arrowhead matrix \eqref{ah} is unitarily irreducible and dichotomous if and only if it is as described by Proposition~\ref{p:dich} with
\eq{abc} a_1,\ldots,a_{n-1} \text{ pairwise distinct and } b_j,c_j\neq 0 \text{ whenever \eqref{offdiag} holds},\en   except when $A$ is of type {\rm (b)} with
\eq{K} \im (e^{-i\theta}A):=K=\begin{bmatrix}k_1 & 0 & \ldots & 0 & m_1 \\
	0 & k_2 & \ldots & 0 & m_2 \\
	\vdots & \vdots & \ddots & \vdots & \vdots \\
	0 & 0 & \ldots & k_{n-1} & m_{n-1} \\
	\overline{m_1} & \overline{m_2} & \ldots & \overline{m_{n-1}} & k_n \end{bmatrix} \en
such that $m_i/\alpha\in\R$ (where $i$ is determined by the position of the entry $t$ in \eqref{opah}) and either \\
{\rm (b')} all $a_j$ ($j\neq i,n$) lie on the same line (equivalently: all of the respective $p_j$ in \eqref{opah} take the same value $p_0\in\{0,1\}$), and
\begin{equation*}
k_{n}=\lambda-\sum_{j\neq i,n}\frac{\abs{m_j}^2}{\lambda-k_j}\text{ with } \lambda=k_i{+(p_0-t)}\frac{m_i}{\alpha},
\end{equation*}
or \\ {\rm (b'')}\ $n=4$, $p_{\sigma(1)}=0$, $p_{\sigma(2)}=1$ for some permutation $\sigma$ of $\{1,2,3\}$ satisfying $\sigma(3)=i$,
while
\eq{k123} k_i=(1-t)k_{\sigma(1)}+tk_{\sigma(2)}, \en
{\eq{m0}(1-t)m_{\sigma(1)}^2=tm_{\sigma(2)}^2\text{ if } m_i+(k_{\sigma(1)}-k_{\sigma(2)})\alpha=0,\en}
and
\eq{det} \det\begin{bmatrix} k_{\sigma(1)}-k_i+t{m_i}/{\alpha} & 0 & m_{\sigma(1)}\\
0 & k_{\sigma(2)}-k_i-(1-t){m_i}/{\alpha} & m_{\sigma(2)} \\
\overline{m_{\sigma(1)}} & \overline{m_{\sigma(2)}} & k_4-k_i-(1-2t){m_i}/{\alpha}\end{bmatrix}=0.\en
\end{thm}

\begin{proof}
An arrowhead dichotomous matrix $A$ meets the description provided by Proposition~\ref{p:dich}. Let $S$ denote the set of indices $j$ for which \eqref{offdiag} holds.
In other words, $S=\{1,\ldots,n-1\}$ in case (a) and  $S=\{1,\ldots,n-1\}\setminus\{i\}$ in case (b).		
		
If one of $b_j,c_j$ equals zero for some $j\in S$, then the other one is also zero, due to the first condition in \eqref{offdiag}. Proposition \ref{p:redsuf}(i) would then imply that $A$ has a normal eigenvalue. Moreover, \eqref{offdiag} also implies that
$b_j\overline{c_k}=b_k\overline{c_j}$ for $j,k\in S$. So, if $a_j=a_k$, then $A$ would have a normal eigenvalue due to  Proposition \ref{p:redsuf}(ii).

We have established that, if $A$ is unitarily irreducible, then indeed $b_j,c_j\neq 0$ and $a_j$ are all distinct for $j\in S$. Comparing \eqref{diag} with \eqref{coni}, we see that in case (b) also $a_i,a_n$ are different from $a_j$ for $j\in S$. Therefore, all $a_1,\ldots,a_{n-1}$ are pairwise distinct.

In what follows, we {therefore} suppose that \eqref{abc} holds (this implies that \eqref{offdiag} holds for all $j\in S$). {Recall that, by \eqref{opah} and \eqref{AP},  $\re(e^{-i\theta}A)$ has zero off-diagonal entries in the $j$th row and column for all $j\in S$.}  This implies in particular that $m_j\neq 0$ for $j\in S$.
Also, when $\re(e^{-i\theta}a_j)=\re(e^{-i\theta}a_\ell)$, we have $k_j\neq k_\ell$.

Observe also that instead of treating unitary (ir)reducibility of $A = e^{i\theta}\left((h_0I + (h_1-h_0)P+iK\right)$, for $P$ defined in Lemma \ref{l:op}, we may consider the question of whether $K$ and $P$ have a common invariant subspace.
Replacing $K$ by $K+\epsilon P$ with a sufficiently small real $\epsilon\neq 0$ if needed, we may suppose that all $m_j$, including $m_i$ in case (b), are different from zero, while all $k_j$ ( $j=1,\ldots,n-1$) are mutually distinct.
Then the eigenvalues of $K$ are all simple and different from $k_1,\ldots,k_{n-1}$. Applying Lemma~\ref{l:eig} to $K$ in place of $A$, we rewrite \eqref{inter} and \eqref{eig} as
\eq{xn'} \sum_{j=1}^{n-1}\frac{\abs{m_j}^2}{\lambda-k_j}+k_n-\lambda=0, \en
and
\eq{xjn'} x_j=\frac{m_j}{\lambda-k_j}\neq 0, \quad j=1,\ldots, n-1;\ x_n=1, \en
respectively. Due to \eqref{xjn'} the eigenvectors of $K$ (and thus the reducing eigenvectors of $A$, if any) have all their coordinates different from zero. In other words, they are not orthogonal to any $e_j$.
This makes Lemma~\ref{l:red} applicable.

We now treat several possible situations separately.

{\bf I.} {\sl $P$ is diagonal.}

Applying a permutational similarity (not involving the last row and column so that the arrowhead structure of $K$ is preserved) we may without loss of generality suppose that $P$ is either $0_s\oplus I_{n-s}$ or $I_s\oplus 0_{n-s}$ for some $s$ and $n$ with $1\leq s<n$.

If $L$ is a common invariant subspace of $K$ and $P$, it is also invariant under the action of  \[ D= \diag[k_1,\ldots,k_s,0,\ldots,0].\] Choose $x\in L$ which is an eigenvector of $K$. All its coordinates are different from zero, as observed earlier. Applying $D$ to $x$ repeatedly, we see that, along with $x$, the subspace $L$
also contains all the vectors $[k_1^jx_1,\ldots,k_s^jx_s,0,\ldots,0]^T$, $j=1,2,\ldots$.  Since $k_1,\ldots,k_s$ are mutually distinct, the Vandermonde determinant gives us that  the first $s$ such vectors are linearly independent. Hence the span of these vectors has dimension $s$ and thus contains $e_1,\ldots,e_s$.  Consequently, $e_1,\ldots,e_s\in L$, and by Lemma~\ref{l:red}, $L$ is the whole space. We have proved that $A$ is unitarily irreducible, which agrees with the statement of the theorem, and thus concludes case {\bf I}.\\

{\bf II.}  {\sl $P$ is of the form \eqref{opah}.}

To simplify the notation, we use the transposition similarity switching the rows and columns numbered $i$ and $n-1$. In other words,
without loss of generality
\eq{H1} P=\diag[p_1,\ldots,p_{n-2}]\oplus\begin{bmatrix} t & \alpha \\ \overline{\alpha} & 1-t\end{bmatrix}, \en
each $p_j$ being equal 0 or 1,  while $K$ is still given by \eqref{K}.

{\bf II.b'.} {\sl All $p_j$ in \eqref{H1} are the same.}

Passing from $P$ to $I-P$ allows us to switch between the cases
$p_1=\cdots=p_{n-2}=0$ and $p_1=\cdots=p_{n-2}=1$, while leading to the change $t\mapsto 1-t$, $\alpha\mapsto -\alpha$
in the second direct summand of \eqref{H1}. Since condition (b') from the statement of Theorem \ref{th:diur} is invariant under the latter, we may suppose that
\[ P=0_{n-2}\oplus \begin{bmatrix} t & \alpha \\ \overline{\alpha} & 1-t\end{bmatrix}.\]

Let us determine when $A$ has a reducing eigenvector, i.e., a common eigenvector of $P$ and $K$.
Such an eigenvector cannot correspond to the (simple) eigenvalue $1$ of $P$, since it is collinear with $y=[0,\ldots,0,\alpha,1-t]^T$ and thus orthogonal to $e_1$.

The zero eigenvalue of $P$, however, has multiplicity $n-1$, and leaves more room to maneuver. The respective eigenvectors either have two last entries equal to zero, in which case they are not eigenvectors of $K$, or can be normalized by setting their last entry equal to $1$ and then written as
\eq{eH} [x_1,\ldots,x_{n-2},-\alpha/t,1]^T \en  with $x_1,\ldots,x_{n-2}$ arbitrary.

We now switch our focus to $K$.  From \eqref{xjn'} and \eqref{xn'}, considering the convention $i=n-1$, we see that there exists an eigenvector of $K$ of the form \eqref{eH} if and only if {(b')} holds (with $\lambda$ being the respective eigenvalue).

We have addressed the situation of common eigenvectors of $K$ and $P$. It remains to show that, when (b') does not hold, $K$ and $P$ do not have any non-trivial common invariant subspaces. To this end, a non-trivial common invariant subspace $L$, or its orthogonal complement $L^\perp$, must contain an eigenvector $y$ corresponding to the simple eigenvalue (namely, $1$) of $P$. Without loss of generality, let $y\in L$. In addition to $y$ and the vector
\[ Ky=[m_1(1-t),\ldots,k_{n-1}\alpha +m_{n-1}(1-t),*]^T,\]  $L$ also contains their linear combination $y_1=[m_1,\ldots,m_{n-2},*,0]^T$.

In turn, there is a linear combination of $Ky_1$ and $y$ of the form $[k_1m_1,\ldots,k_{n-2}m_{n-2},*,0]^T$, thus also lying in $L$. Repeating sufficiently many times, we see that $L$ contains
vectors of the form \[ [k_1^jm_1,\ldots,k_{n-2}^jm_{n-2},*,0]^T \text{ for all } j=1,2,\ldots. \]
Therefore, $\dim L\geq n-2$, and since $y$ is also in $L$ and has a nonzero last coordinate, we conclude that $\dim L\geq n-1$.  Consequently $\dim L^{\perp}\leq 1$. But $\dim L^{\perp}=1$ is not an option, since otherwise $L^\perp$ would be spanned by a common eigenvector of $P$ and $K$. So $L$ is the whole space. This concludes case {\bf II.b'}.\\

{\bf II.b''}. {\sl Both $0$ and $1$ are actually present on the diagonal of the first block in \eqref{H1}.}

All eigenvectors of $P$ then have at least one zero entry, while according to \eqref{xjn'} the eigenvectors of
$K$ have none. Therefore, $A$ has no reducing eigenvectors, and its reducing subspaces are at least 2-dimensional.
From here it immediately follows that $A$ is unitarily irreducible when $n=3$.

Now let $n\geq 4$ and suppose that there exists a reducing subspace $L$ of $A$ with dimension at least 3. Then it contains a non-zero vector $z$ with the zero coordinates $z_{n-1}, z_n$. At least one of the vectors $Pz$, $(I-P)z$ is also different from zero. Without loss of generality $y:=(I-P)z\neq 0$. A direct computation shows that then $PKy$ is a multiple of the vector $[0,\ldots,0,\alpha,1-t]^T$. So either \eq{vil} [0,\ldots,0,\alpha,1-t]^T\in L,\en  or $PKy=0$.
	
In the latter case $(I-P)Ky=Ky$. Since the last two entries of $(I-P)Ky$ form a vector collinear with $[-\alpha,t]^T$ while the $(n-1)$st coordinate of $Ky$ is zero, this can only happen if
\[ Ky=[k_1y_1,\ldots, k_{n-2}y_{n-2},0,0]^T. \]
	
Since $Ky\in L$ along with $y$, by a repeated application of this reasoning we will either at some step arrive at \eqref{vil} or conclude that
\[  [k_1^my_1,\ldots, k_{n-2}^my_{n-2},0,0]^T\in L, \quad m=1,2,\ldots.  \]
From all $k_j$ being distinct, it then follows that $e_j\in L$ for all $j$ corresponding to non-zero values of $y_j$ and so $L$ is the whole space by Lemma~\ref{l:red}.  	
	
On the other hand, if \eqref{vil} holds, then eigenvectors $x$ of $K$ lying in $L^\perp$ must thus satisfy
\[ \overline{\alpha}x_{n-1}+(1-t)x_n=0. \] Due to \eqref{xjn'}, this condition can hold only for at most one linearly independent eigenvector of $K$. Since $L^\perp$ cannot be one-dimensional, we conclude that it is trivial, and again $L$ is the whole space.

We have thus shown that non-trivial reducing subspaces of $A$ may only be 2-dimensional. If $n>4$, any such subspace will have a complement which is a reducing subspace of dimension at least 3, leading to a contradiction. Therefore, for $n>4$, $A$ must be unitarily irreducible.
{So, unitary reducibility in the II.b'' setting can possibly occur only if $n=4$. Let us concentrate on this case.}

Switching the first and the second row and column of $A$ if necessary, in place of \eqref{H1} we get
\begin{equation*}
P=\diag[0,1]\oplus\begin{bmatrix} t & \alpha \\ \overline{\alpha} & 1-t \end{bmatrix}.
\end{equation*}
We may also apply a diagonal unitary similarity to $A$, allowing us to suppose that $m_j>0$ for all $j$; the argument of $\alpha$ under this similarity will change respectively, turning
condition $m_3/\alpha\in\R$ (from the statement of this theorem) simply into $\alpha\in \R$.
The permutation $\sigma$ in (iv) then becomes redundant, and conditions \eqref{k123}--\eqref{det} simplify respectively as
\eq{k123'} k_3=(1-t)k_1+tk_2, \en
{ \eq{mo'} (1-t)m_1^2=tm_2^2 \ \text{ if }\  m_3+(k_1-k_2)\alpha=0,\en }
 and
\eq{det'} \det\begin{bmatrix} k_1-k_3+t{m_3}/{\alpha} & 0 & m_1\\
0 & k_2-k_3-(1-t){m_3}/{\alpha} & m_2 \\
{m_1} & {m_2} & k_4-k_3-(1-2t){m_3}/{\alpha}\end{bmatrix}=0.\en
Observe also that $P$ has a common invariant subspace with $K$ if and only if it has one with
\begin{equation*}
	K'=K-k_1I+(k_1-k_2)P=\begin{bmatrix} 0 & 0 & 0 & m_1 \\ 0 & 0 & 0 & m_2 \\
0 & 0 & k' & m \\ {m_1} & {m_2} & \overline{m} &  k\end{bmatrix},
	\end{equation*}
where \eq{k} k'=k_3-k_1+(k_1-k_2)t=k_3-k_1(1-t)-k_2t, \en
\eq{m} m= m_3+(k_1-k_2)\alpha, \en and
\eq{k4}
	 k=k_4-k_1+(k_1-k_2)(1-t)=k_4-k_1t-k_2(1-t).
\en

{\sl Necessity.} Suppose $A$ has a 2-dimensional reducing subspace $L$. {We will compute necessary conditions for this subspace $L$ to exist.} Being invariant under $K$, it is spanned by two of its eigenvectors (recall that the eigenvalues of $K$ are distinct). From \eqref{xjn'}
we see that, when the fourth coordinate of these vectors is chosen equal to one, their difference (also lying in $L$) has real non-zero first three entries while the fourth is equal to zero. Scaling, we find
a vector $y=[y_1,y_2,1,0]^T\in L$ with non-zero real $y_1,y_2$.

Being invariant under $P$, the subspace $L$ then also contains the (linearly independent) vectors $Py=[0,y_2,t,\overline{\alpha}]^T$ and $(I-P)y=[y_1,0,1-t,-\overline{\alpha}]^T$, which thus form a basis of $L$. Since \[K'y=[0,0,k',m_1y_1+m_2y_2+{\overline{m}}]^T\in L,\]
the fact that $L$ is 2-dimensional implies that $K'y=0$. In other words, $k'=0$, which from \eqref{k} is equivalent to \eqref{k123'}, and
\eq{m123} {m_1}y_1+{m_2}y_2+\overline{m}=0. \en
Since $m_j$ and $y_j$ are real, \eqref{m123} may hold only if $m\in \R$. By \eqref{m}, this implies that $\alpha\in\R$.

Furthermore, the vector \[ \frac{1}{\alpha}K'Py =\left[m_1,m_2,m, k+({m_2}y_2+{m}t)/\alpha\right]^T \]
also lies in $L$ implying that
\eq{my} (1-t)\frac{m_1}{y_1}+t\frac{m_2}{y_2} = m, \en
\eq{ky} \alpha\left(\frac{m_2}{y_2}-\frac{m_1}{y_1}\right)=k+\frac{m_2y_2+mt}{\alpha}. \en

{If $m=0$,  \eqref{m123} and \eqref{my} imply that $(1-t)m_1^2=tm_2^2$. Due to \eqref{m}, this shows the necessity of \eqref{mo'}.
The necessity of \eqref{det'} is established by the following reasoning.}

Along with $K'$ (or $K$) and $P$, the subspace $L$ is also invariant under \[
K-\frac{m_3}{\alpha}P=\begin{bmatrix}k_1 & 0 & 0 & m_1 \\ 0 & k_2-\frac{m_3}{\alpha} & 0 & m_2 \\ 0 & 0 & k_3-\frac{m_3}{\alpha}t & 0 \\
	m_1 & m_2 & 0 & k_4-\frac{m_3}{\alpha}(1-t)\end{bmatrix}.\]
Observe that $e_3$ is an eigenvector of $K-\frac{m_3}{\alpha}P$ corresponding to the eigenvalue $k_3-\frac{m_2}{\alpha}t$. If this eigenvalue is simple, then $e_3$ must lie either in $L$ or its
orthogonal complement $L^\perp$, which is a contradiction with Lemma~\ref{l:red}. Consequently, the matrix obtained from $K-\frac{m_3}{\alpha}P-(k_3-\frac{m_3}{\alpha}t)I$ by deleting its third row and column must be singular. This is exactly condition \eqref{det'}.

{\sl Sufficiency.} {We will now provide sufficient conditions for $A$ to have a 2-dimensional reducing subspace $L$.}  It is enough to show that if $\alpha\in\R$, $m$ is defined by \eqref{m} {and is then real}, and \eqref{k123'}--\eqref{det'} hold, then there exist $y_1,y_2$ satisfying \eqref{m123}--\eqref{ky}. Indeed, given such $y_1$ and $y_2$, let $y=[y_1,y_2,1,0]^T$. The subspace $L=\Span\{Py,(I-P)y\}$ is 2-dimensional and $P$-invariant by construction. Moreover, $K'y=0$ due to \eqref{k123'} and \eqref{m123}, while ${\frac{1}{\alpha}}K'Py=\frac{m_1}{y_1}(I-P)y+\frac{m_2}{y_2}Py\in L$ due to \eqref{my} and \eqref{ky}.  So, $L$ is $K'$- (and thus $K$-) invariant as well, implying that it is a non-trivial reducing subspace of $A$.

To address the existence of such $y_1$ and $y_2$, first consider the case $m=0$. Condition \eqref{mo'} implies that \eqref{m123} and \eqref{my} hold for any pair  $y_1\neq 0$, $y_2=-m_1y_1/m_2$. Plugging the expression for $y_2$ in terms of $y_1$  into \eqref{ky}  and solving for $y_1$, we can thus {have a more restrictive condition for $y_1$ and $y_2$ that satisfies \eqref{ky} as well}. Since $m=0$, condition \eqref{det'} holds automatically; indeed the first and the second diagonal entries of the matrix in question are both equal to zero due to \eqref{k123'} and \eqref{mo'}.

If $m\neq 0$, condition \eqref{mo'} is {irrelevant}. On the other hand, the system $\{$\eqref{m123},\eqref{my}$\}$ is consistent with no restrictions imposed on $m_j$, and the solutions are real. Therefore, there is a unique value of $k_4$ for which \eqref{k4} and \eqref{ky} also hold, thus making $A$ unitarily reducible. The argument from the necessity proof above implies that this must be the same value of $k_4$ as the one which satisfies \eqref{det'}. When $m\neq 0$, condition \eqref{det'} defines $k_4$ uniquely, and therefore it indeed implies \eqref{ky}. This completes the proof. \end{proof}

\section{Gau--Wu numbers of some arrowhead matrices}\label{s:arrow}

According to \cite[Theorem 7]{CaRaSeS}, a unitarily irreducible $n\times n$ matrix $A$ has Gau--Wu number $k(A)=n$
if and only if it is dichotomous. The results of Section~\ref{s:irred} therefore immediately imply

\begin{thm}\label{th:max}
An arrowhead matrix \eqref{ah} is unitarily irreducible with the Gau--Wu number $k(A)=n$ if and only if it is as described in Proposition~\ref{p:dich}, where, in addition, $a_j$ are distinct and $b_j,c_j$ in \eqref{offdiag} are non-zero{, and we also are not in either of the situations {\em b', b''}} of Theorem~\ref{th:diur}.
\end{thm}

Condition \eqref{offdiag} of Proposition~\ref{p:dich} guarantees the ``diagonalization after rotation''
of the Hermitian part of $A$, while condition \eqref{coni} is responsible for the fact that the resulting matrix has at most two distinct eigenvalues. A more detailed analysis shows that $k(A)$ can be computed explicitly under condition \eqref{offdiag} alone (provided that it holds for all $j=1,\ldots,n-1$). Let us first consider the leading case when all $b_j, c_j$ are different from zero.

\begin{thm}\label{th:dak}
Let an arrowhead matrix \eqref{ah} satisfy \eqref{offdiag} with $b_j,c_j\neq 0$; $j=1,\ldots,n-1$. Then $k(A)$ equals the sum of multiplicities of the largest and the smallest eigenvalues of
$\im (e^{-i\theta}A)$.
\end{thm}

\begin{proof} Multiplying $A$ by an appropriate non-zero scalar, we may without loss of generality suppose that in \eqref{offdiag} {$\theta=0$}, i.e., $c_j=-\overline{b_j}$. The matrix $H:=\re A
$ is then diagonal. Denoting by $J\subseteq\{1,\ldots,n\}$ the set of indices corresponding to the extremal (maximal and minimal)
diagonal entries of $H$, we observe that $f_A(e_j)$ lie on the vertical supporting lines $\ell_1,\ell_2$ of $W(A)$, and thus on its boundary, for all $j\in J$.  So, $k(A)\geq k$, where $k$ is the cardinality of $J$.

Before proving the reverse inequality, we observe that an appropriate diagonal unitary similarity allows us to adjust the arguments of all $b_j$ to the value $\pi/2$, thus guaranteeing $c_j=b_j=im_j$ with $m_j>0$ for $j=1,\ldots,n-1$; here we let $K=\im A$ and define $m_j$ as in equation \eqref{K}. Let $h_j$ and $k_j$ respectively denote the diagonal entries of $H$ and $K$.

For every $t>0$, the interlacing theorem gives us that the maximal eigenvalue $\xi_+(t)$ of the matrix
\eq{htk} H+tK=\begin{bmatrix}h_1+tk_1 & 0 & \ldots & 0 & tm_1 \\
0 & h_2+tk_2 & \ldots & 0 & tm_2 \\
\vdots & \vdots & \ddots & \vdots & \vdots \\
0 & 0 & \ldots & h_{n-1}+tk_{n-1} & tm_{n-1} \\
tm_1 & tm_2 & \ldots & tm_{n-1} & h_n+tk_n \end{bmatrix} \en
is strictly bigger than the diagonal entries $h_j+tk_j$, $j=1,\ldots,n-1$. Formulas \eqref{eig} for the matrix \eqref{htk} take the form
\[ x^+_j(t)=\frac{t^2m_j^2}{\xi_+(t)-h_j-tk_j},\ j=1,\ldots,n-1;\quad x^+_n=1, \]
implying that all the entries of the respective eigenvector $x(t)$ are positive.
So $\scal{x^+(t_1),x^+(t_2)}>0$ for any $t_1,t_2>0$.

Similarly, the minimal eigenvalue $\xi_-(t)$ of the matrix \eqref{htk} is strictly smaller than $h_j+tk_j$, $j=1,\ldots,n-1$, the respective eigenvector $x^-(t)$ has the entries
\[ x^-_j(t)=\frac{t^2m_j^2}{\xi_-(t)-h_j-tk_j},\ j=1,\ldots,n-1;\quad x^-_n=1, \]
the first $n-1$ of them being negative, and so $\scal{x^-(t_1),x^-(t_2)}>0$ for any $t_1,t_2>0$ as well.

Denote by $L_\pm$ the eigenspace of $H$ corresponding to its maximal/minimal eigenvalue, and let  $X\subseteq\C^n$ be an orthonormal set, the elements of which are mapped by $f_A$ into $\partial W(A)$. This set contains at most one (normalized) vector from the family $x^+(t)$, $t>0$, and similarly at most one from the family $x^-(t)$, $t>0$,
with the rest lying in ${L_+}\cup {L_-}$.

Suppose $X$ consists of more than $k$ vectors. Then at least $k-1$ of them lie in ${L_+}\cup { L_-}$ and, since $\dim{ L_+}+\dim{L_-}=k$,  the set $X$ must contain a basis of ${L_+}$ or  ${L_-}$ (or both). Any other vector of $X$ then has to be orthogonal to this eigenspace, and therefore to at least one standard basis vector. Since all the entries of $x^\pm(t)/\norm{x^\pm(t)}$ are non-zero, neither of these vectors may lie in $X$. Thus $X\subseteq{L_+}\cup {L_-}$, and as such $X$ cannot contain more than $k$ elements. This contradiction completes the proof.
\end{proof}

From the proof of Theorem~\ref{th:dak} it is clear that any set $X$ of $k(A)$ orthonormal vectors satisfying $f_A(X)\subseteq\partial W(A)$ is the union of bases of $L_+$ and $L_-$.

Suppose now that there are some zero off-diagonal pairs $\{b_j,c_j\}$ in \eqref{ah}; denote the set of respective indices $j$ by $J_0$ and its complement to $\{1,\ldots,n-1\}$ by $J_1$. Then a permutational similarity
reduces $A$ to the direct sum of the diagonal matrix $A_0$ with the diagonal entries $a_j, j\in J_0$, and an arrowhead matrix $A_1$ obtained from $A$ by deleting its rows and columns numbered by $j\in J_0$.
In this situation, the following result of Lee \cite{Lee} (see Proposition 3.1 there, or \cite[Theorem 2]{CaRaSeS}) applies.  For convenience of reference, we change the notation slightly from what was used in \cite{Lee,CaRaSeS}.

\begin{prop}\label{p:dirsum}
Let $A$ be unitarily similar to $A_0\oplus A_1$, with $A_0$ being normal. Then $ k(A)=k'(A_0)+k'(A_1)$.
\end{prop}

Here $k'(A_j)$ stands for the maximal number of orthonormal vectors whose images under $f_{A_j}$ lie in $\partial W(A)\cap\partial W(A_j)$, $j=0,1$.

Though not stated in \cite{Lee} explicitly, from the proof outline of Proposition~\ref{p:dirsum} there it can be seen that $k'(A_0)$ is nothing but the number of the eigenvalues of $A_0$ (counting their multiplicities) lying in $\partial W(A)$.  The following theorem is therefore an immediate corollary of Theorem~\ref{th:dak}.

\begin{thm}\label{th:dako}Let an arrowhead matrix \eqref{ah} satisfy \eqref{offdiag}, with $S=\{j\colon b_j=c_j=0\}$. Denote by $A_0$ the (diagonal) submatrix formed by the rows and columns of $A$ corresponding to elements of $S$, and let $A_1$ denote the (arrowhead) submatrix formed by the remaining rows and columns of $A$.  Then $$k(A)=k'(A_1)+\#\{a_j\mid j\in S, a_j\in \partial W(A)\}.$$
\end{thm}

Note that according to \eqref{offdiag} all off-diagonal pairs $\{b_j,c_j\}$, except maybe one, of a dichotomous arrowhead matrix have to be ``balanced.''
It is worth mentioning that in the case when an arrowhead matrix is not dichotomous, and its off-diagonal pairs are all ``unbalanced'' in the same way, $k(A)$ attains its minimal possible value 2.

\begin{thm}\label{th:arrow}
Let $a_1,\ldots,a_{n-1}$ in \eqref{ah} be distinct, $C$ denote their convex hull, and $J_0=\{j\mid a_j\in \partial C\}$.  Assume that either (a) $\abs{b_j}\geq \abs{c_j}$ for all $j=1,\ldots,n-1$ or (b) $\abs{c_j}\geq \abs{b_j}$ for all $j=1,\ldots,n-1$.  Furthermore, assume that these inequalities are strict for all $j\in J_0$.  Then $k(A)=2$.
\end{thm}
Conditions of Theorem~\ref{th:arrow} hold in particular when $c_j=0, b_j\neq 0$ for all $j=1,\ldots n-1$. This is the case of so-called {\em pure almost normal} matrices for which the result was established in \cite{CaRaSeS}.

\begin{proof}
The real part of the matrix $e^{i\theta}A$ is
\begin{equation*}
\re\left(e^{i\theta}A\right) = \frac{1}{2}
\left[\begin{array}{ccccc}
2\re\left( a_1 e^{i\theta} \right)& 0 & \ldots & 0 & e^{-i\theta}\overline{c_1} + e^{i\theta}b_1\\
0 & 2\re\left( a_2 e^{i\theta} \right) & \ldots & 0 & e^{-i\theta}\overline{c_2} + e^{i\theta}b_2\\
\vdots & \vdots & \ddots & \vdots &\vdots\\
0 & 0 & \ldots & 2\re\left(a_{n-1} e^{i\theta}\right)& e^{-i\theta}\overline{c_{n-1}} + e^{i\theta}b_{n-1}\\
e^{i\theta}c_1 + e^{-i\theta}\overline{b_1}& e^{i\theta}c_2 + e^{-i\theta}\overline{b_2} & \ldots & e^{i\theta}c_{n-1}+e^{-i\theta}\overline{b_{n-1}} & 2\re\left( a_n e^{i\theta} \right)
\end{array}
\right].
\end{equation*}
	
Since $|b_j|\neq |c_j|$ for all $j\in J_0$, the off-diagonal entries in rows and columns corresponding to elements of $J_0$ are non-zero {for all values of $\theta$}.  Therefore, the determinant of the matrix obtained by only keeping the $j$th ($j\in J_0$) and $n$th rows and columns of the matrix $\re(e^{i\theta}A-a_je^{i\theta}I_n)$ is negative.  Due to the interlacing theorem, the maximal eigenvalue $\xi(\theta)$ of $\re(e^{i\theta}A)$ is strictly bigger than $\re(a_je^{i\theta})$ for all $j\in J_0$.   For $i\in J_1$, we have $\re(a_ie^{i\theta})<\re(a_je^{i\theta})$ for some $j\in J_0$ by definition of $J_0$; hence $\re(a_ie^{i\theta})<\xi(\theta)$ in this situation as well.  Consequently, the first $n-1$ diagonal entries of $\re(e^{i\theta}A)-\xi(\theta)I_n$ are all non-zero, implying that $\xi(\theta)$ is not an eigenvalue of the matrix formed by the first $n-1$ rows and columns of $\re(e^{i\theta}A)$. Thus the interlacing theorem tells us that $\xi(\theta)$ is a simple eigenvalue of $\re(e^{i\theta}A)$.
	
This means that each support line of $W(A)$ contains exactly one point of $\partial W(A)$.  Moreover, for the line lying to the right of $W(A)$ and forming the angle $\pi/2-\theta$ with the positive real axis, this point is generated by the unit eigenvector of $\re(e^{i\theta}A)$ corresponding to $\xi(\theta)$. Lemma \ref{l:eig} shows that this eigenvector is collinear with the vector

\[
{x}(\theta) =
	\left[\frac{\left(e^{-i\theta}\overline{c_1} + e^{i\theta}b_1\right)}{2\left(\xi(\theta)-\re\left( a_1 e^{i\theta} \right)\right)},\ldots,
	\frac{\left(e^{-i\theta}\overline{c_{n-1}} + e^{i\theta}b_{n-1}\right)}{2\left(\xi(\theta)-\re\left( a_{n-1} e^{i\theta} \right)\right)},1\right]^T.
\]
	
The scalar product of two such vectors corresponding to the angles $\theta_1$ and $\theta_2$ is
	
\eq{sp}
\scal{
{x}(\theta_1),
{x}(\theta_2)}=\sum_{j=1}^{n-1}\frac{e^{i(\theta_1-\theta_2)}|b_j|^2+2\re\left(e^{i(\theta_1+\theta_2)}b_jc_j\right)+e^{i(\theta_2-\theta_1)}|c_j|^2}{4(\xi(\theta_1)-\re( a_j e^{i\theta_1}))(\xi(\theta_2) - \re( a_j e^{i\theta_2}))}+1
\en
	
In order for this scalar product to equal zero, in particular the imaginary part of the righthand side of \eqref{sp} must equal zero. The denominator has no imaginary part, so we need only consider the part of the numerator that may not be purely real:
\begin{equation*}\label{sp2}
	e^{i(\theta_1-\theta_2)}|b_j|^2 + e^{i(\theta_2-\theta_1)}|c_j|^2.
\end{equation*}
Writing $e^{i(\theta_1-\theta_2)} = s + it$, we have
\begin{equation*}
	\im\left(e^{i(\theta_1-\theta_2)}|b_j|^2 + e^{i(\theta_2-\theta_1)}|c_j|^2\right) = t(\abs{b_j}^2 -\abs{c_j}^2).
\end{equation*}
Since $|b_j|\geq |c_j|$ for all $j$ or $|c_j|\leq |b_j|$ for all $j\in J_0\cup J_1$, with strict inequality for at least one $j\in J_0$, the imaginary part of \eqref{sp} will be zero only when $t = 0$. This occurs when $\theta_1 - \theta_2 = \pi \mod 2\pi$.
Therefore, only two mutually orthogonal vectors generating boundary points of $W(A)$ can be chosen simultaneously.
\end{proof}

\section{Gau--Wu numbers of $4\times4$ matrices}\label{s:4x4}

The Gau--Wu number of $2\times2$ matrices is always 2, and the Gau--Wu number of $3\times3$ matrices was completely classified by Wang and Wu \cite{WangWu13}.  Studying the Gau--Wu numbers of $4\times4$ matrices is the natural next step.  Our Proposition \ref{p:dirsum} and Theorem \ref{th:dako} allow us to immediately determine the Gau--Wu number for any unitarily reducible $4\times4$ matrix, which we will illustrate at the beginning of this section.
In turn, unitarily irreducible matrices will be treated later in the section.

The following terminology will be useful.  We define the \textit{base polynomial} $F_A(x:y:t)$ to be the quantity $\det\left(x(\textmd{Re }A)+y(\textmd{Im }A)+tI\right)$, where $\textmd{Re }A$ and $i\textmd{Im }A$ are the Hermitian and skew-Hermitian parts of $A$, and the \textit{boundary generating curve} $\Gamma_{F_A}^\wedge$ to be the dual to the curve $F_A(x:y:t)=0$ \cite{CaDeRaSeSpY, ChiNa12}.  Figure \ref{fig:ellipses} gives the full range of Gau--Wu numbers for matrices that are unitarily reducible to a direct sum of two unitarily irreducible $2\times2$ blocks which then individually have elliptical numerical ranges.  Note that Corollary 19 of \cite{CaDeRaSeSpY} tells us that we will also obtain this type of numerical range for any matrix (even unitarily \textit{irreducible} matrices) with a {base polynomial $F_A$} that factors into two real quadratic polynomials.

\begin{figure}[h]
\begin{subfigure}{.33\textwidth}
  \centering
  \includegraphics[width=.8\linewidth]{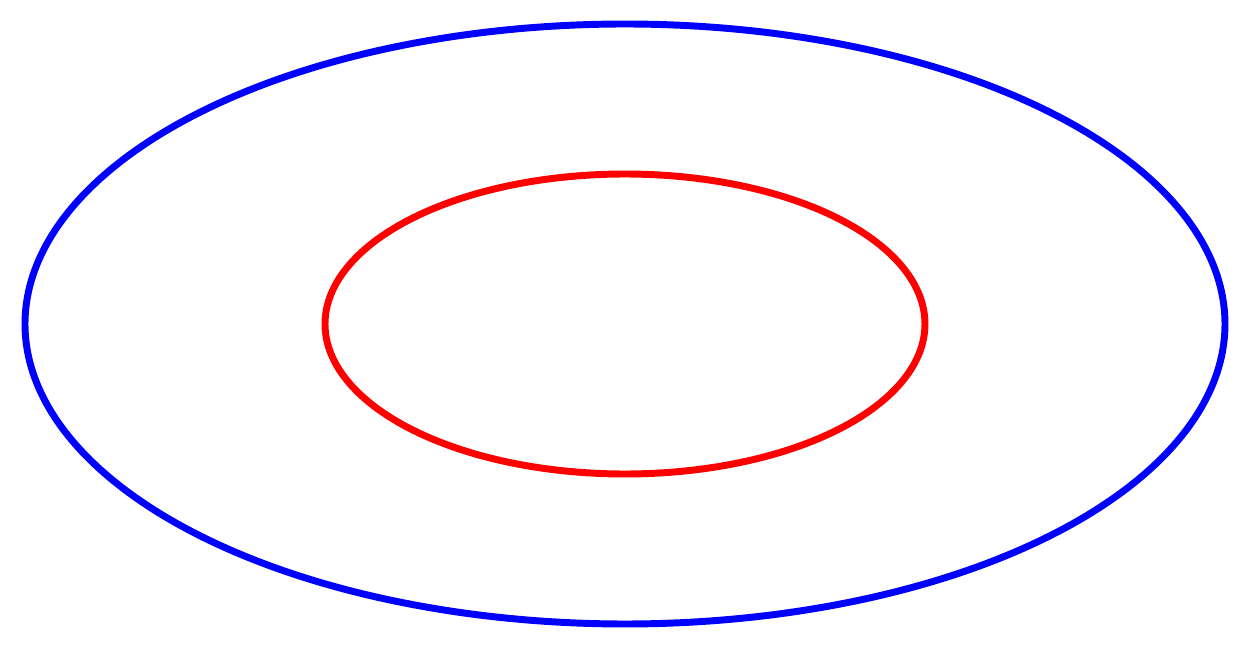}
  \caption{$k(A)=2$}
  \label{fig:sfig1a}
\end{subfigure}%
\begin{subfigure}{.33\textwidth}
  \centering
  \includegraphics[width=.8\linewidth]{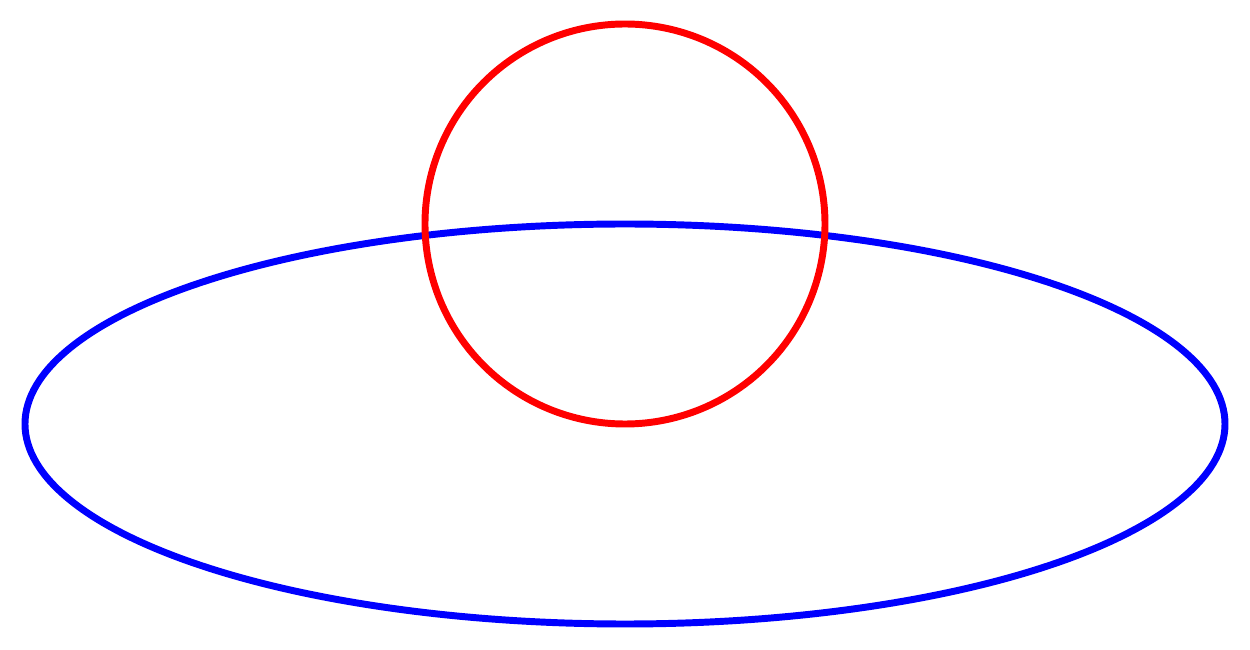}
  \caption{$k(A)=3$}
\end{subfigure}
\begin{subfigure}{.33\textwidth}
  \centering
  \includegraphics[width=.8\linewidth]{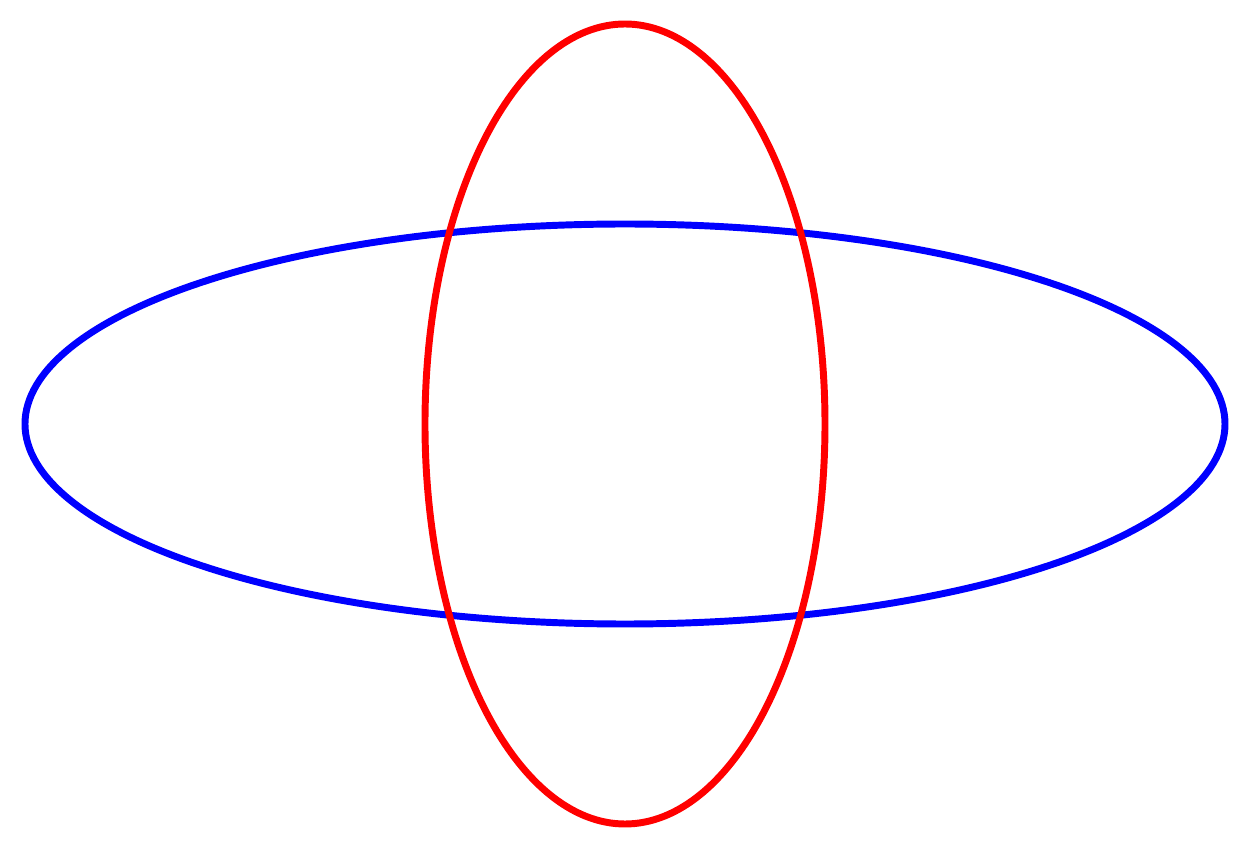}
  \caption{$k(A)=4$}
\end{subfigure}
\caption{Examples of configurations of boundary generating curves for direct sums of pairs of unitarily irreducible $2\times2$ matrices, with various Gau--Wu numbers.}
\label{fig:ellipses}
\end{figure}

For a matrix which unitarily reduces to one unitarily irreducible $2\times2$ block, and two scalar blocks, some of the primary configurations of the numerical range are illustrated in Figure \ref{fig:ellipselines}.  Note that this situation also occurs for any matrix (possibly unitarily irreducible) whose base polynomial factors into two real linear terms and one real quadratic polynomial.

\begin{figure}[h]
\begin{subfigure}{.5\textwidth}
  \centering
  \includegraphics[width=0.528\linewidth]{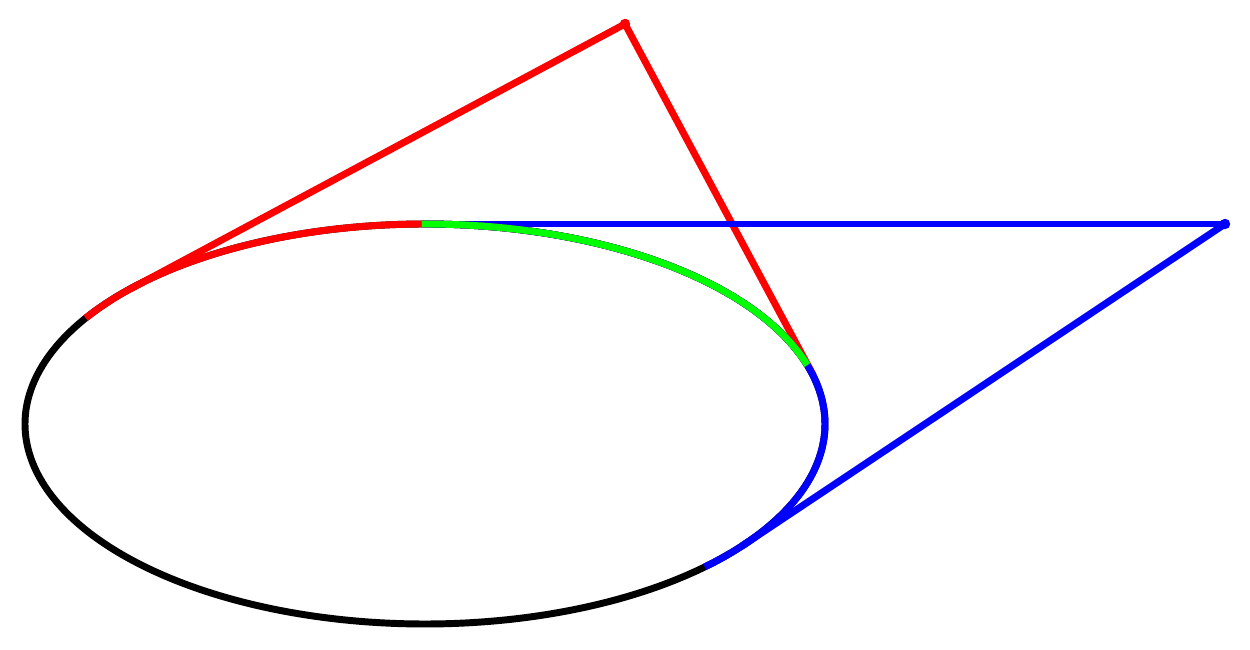}
  \caption{$k(A)=3$}
\end{subfigure}
\begin{subfigure}{.5\textwidth}
  \centering
  \includegraphics[width=0.528\linewidth]{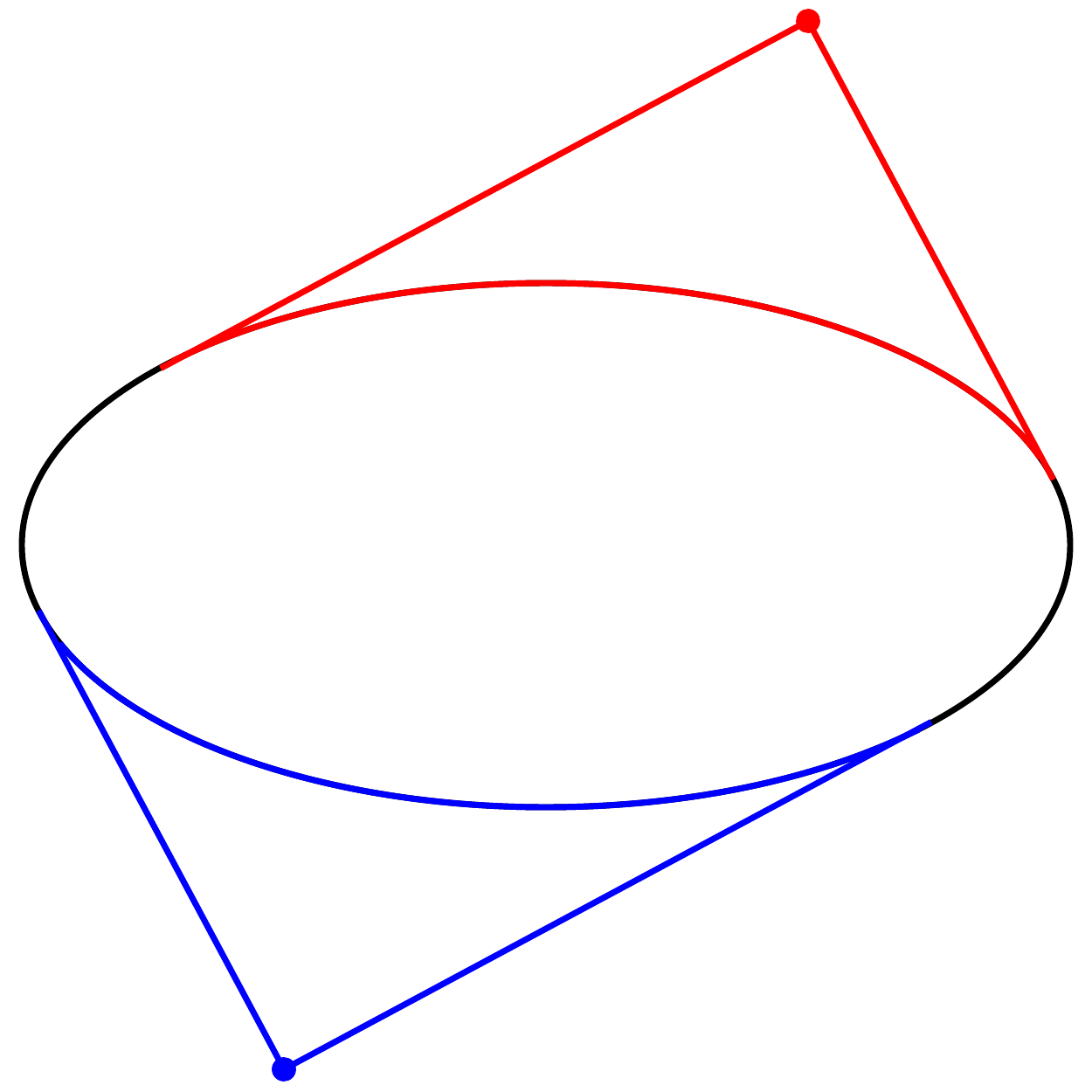}
  \caption{$k(A)=4$}
\end{subfigure}
\caption{Examples of configurations of boundary generating curves for direct sums of a unitarily irreducible $2\times2$ matrix with two scalars, with various Gau--Wu numbers.}
\label{fig:ellipselines}
\end{figure}
	
We observe that a unitarily reducible $4\times4$ matrix with no {singular points on the boundary generating curve}, as in the case of Figure \ref{fig:sfig1a}, will have $k(A)=2$. We formalize this by recalling the notion of {\em seeds} from \cite{CaDeRaSeSpY}.  A seed is a subset of $\partial W(A)$ which is either a flat portion or a singular point of the boundary generating curve. This is equivalent to the notion of the {\em exceptional supporting line} \cite{LLS13} of $W(A)$, a line $\ell$  such that there exist  $z\in W(A)\cap\ell $ with $f^{-1}_A(z)$ containing linearly independent vectors. For matrices $A\in\C^{n\times n}$ of any size $n$, the presence of seeds  (equivalently: the presence of an exceptional supporting line of $W(A)$) implies that $k(A)\geq 3$; see e.g. \cite[Proposition~9]{CaDeRaSeSpY}.  However, unitarily irreducible $A\in\C^{4\times 4}$ with no seeds but $k(A)=3$ do exist. This was observed in \cite[Example~15]{CaDeRaSeSpY} by way of example. Here we treat this phenomenon more systematically.

We will now concentrate on unitarily irreducible matrices. First, let us address the case of maximal possible $k(A)$.

\begin{thm} \label{th:k4} Let $A\in\C^{4\times 4}$. Then $A$ is unitarily irreducible with $k(A)=4$ if and only if, up to a rotation, $A$ is unitarily similar to $H+iK$ (with $H,K$ Hermitian) where either {\rm (a)}
\eq{2+2}   H=\begin{bmatrix} h_1 & 0 & 0 & 0 \\ 0 & h_1 & 0 & 0 \\ 0 & 0 & h_2 & 0 \\ 0 & 0 & 0 & h_2 \end{bmatrix}, \ K=\begin{bmatrix} k_{11} & k_{12} & \sigma_1 & 0 \\ \overline{k_{12}} & k_{22} & 0 & \sigma_2 \\ \sigma_1 & 0 & k_{33} & k_{34} \\
	0 & \sigma_2 & \overline{k_{34}} & k_{44} \end{bmatrix}=:\begin{bmatrix}K_1 & \Sigma \\ \Sigma & K_2\end{bmatrix}, \en
with $\sigma_1>0,\ \sigma_2\geq 0$,  $k_{12} k_{34}\neq 0$ if $\sigma_2=0$, $K_1$ not commuting with $K_2$ if $\sigma_1=\sigma_2$, and at least one of the entries $k_{12}, k_{34}$  is different from zero if $\sigma_1>\sigma_2>0$,

\noindent
or {\rm (b)}
\eq{3+1} H=\begin{bmatrix} h_1 & 0 & 0 & 0 \\ 0 & h_1 & 0 & 0 \\ 0 & 0 & h_1 & 0 \\ 0 & 0 & 0 & h_2 \end{bmatrix}, \ K=\begin{bmatrix} k_1 & 0 & 0 & b_1 \\ 0 & k_2 & 0 & b_2 \\ 0 & 0 & k_3 & b_3 \\
b_1 & b_2 & b_3 & k_4 \end{bmatrix}\in\R^{4\times 4}, \en with pairwise distinct $k_1,k_2,k_3$ and non-zero $b_1,b_2,b_3$. Both in \eqref{2+2} and \eqref{3+1}, $h_1\neq h_2$.
\end{thm}

\begin{proof}
Suppose $A$ is unitarily irreducible and $k(A)=4$. By the definition of $k(A)$, there is an orthonormal basis $\mathcal E$  of $\C^4$ whose image under $f_A$ lies on $\partial W(A)$. According to \cite[Theorem~7]{CaRaSeS}, this image is in fact located on two parallel supporting lines of $W(A)$. Via rotation of $A$, without loss of generality we may suppose that these lines are vertical. Then the matrix of $H:=\re A$ with respect to the basis $\mathcal E$ takes the form given in \eqref{2+2} and \eqref{3+1}.

If $H$ is as in \eqref{2+2}, we use a block diagonal unitary similarity to replace the upper right $2\times2$ block of $K$ by the diagonal matrix $\Sigma$ with the singular values $\sigma_1\geq \sigma_2\geq 0$ of this block on the diagonal, without changing $H$. This is how we arrive at the form $K:=\im A$ in \eqref{2+2}.

For $H$ as in \eqref{3+1}, the upper left $3\times3$ block of $K$ can be diagonalized by a further unitary similarity, leaving $H$ unchanged and thus putting $A$ into an arrowhead form. Yet another (this time, diagonal) unitary similarity can be used to make the entries of the last row and column of $K$ real, and even non-negative, leaving $H$ and the rest of $K$ unchanged.  Thus the form of $K:=\im A$ in \eqref{3+1} emerges.

On the other hand, $H$ as in \eqref{2+2} or \eqref{3+1} guarantees that $k(A)=4$, independent of further restrictions on $K$ and whether or not $A$ is unitarily reducible. It therefore remains to show that these restrictions, as stated in (a) and (b), are equivalent to unitary irreducibility of $A$. In the (b) setting this follows immediately from Theorem~\ref{th:diur}. We thus concentrate on the case of $H,K$ given by \eqref{2+2}.

In this case, $\Sigma=0$ is equivalent to $A$ being normal, and thus implies it is unitarily reducible. Now assume that $\Sigma$ (and thus $\sigma_1$) is nonzero.  Condition $\sigma_2=k_{12}k_{34}=0$ is equivalent to  $H$ and $K$ having a common eigenvector, once again making $A$ unitarily reducible. If $\sigma_1=\sigma_2(\neq 0)$ and $K_1,K_2$ commute, yet another block diagonal unitary similarity can be used to simultaneously diagonalize  $K_1$ and $K_2$, leaving $\Sigma$ unchanged. For the resulting $A$,  $\Span\{e_1,e_3\}$ and  $\Span\{e_2,e_4\}$ are reducing subspaces.

Now we only need to show that $A$ has no 2-dimensional reducing subspaces in the following situations: \\
(i) $\sigma_1>\sigma_2$, and at least one of (both, if $\sigma_2=0$) the entries $k_{12},k_{34}$ is different from zero, and  \\
(ii) $\sigma_1=\sigma_2>0$, $K_1$ and $K_2$ do not commute.

Such a subspace $L$ of $A$  (and its orthogonal complement $L^\perp$), if it exists, should be invariant under all of the following: \eq{inva} \begin{bmatrix}I & 0 \\ 0 & 0\end{bmatrix},\begin{bmatrix}0 & 0 \\ 0 & I\end{bmatrix}, \begin{bmatrix}K_1 & 0 \\ 0 & 0\end{bmatrix}, \begin{bmatrix}0 & 0 \\ 0 & K_2\end{bmatrix}, \begin{bmatrix}0 & \Sigma \\ \Sigma & 0\end{bmatrix}, \begin{bmatrix}\Sigma^2 & 0 \\ 0 & \Sigma^2\end{bmatrix}.\en
In particular, $L$ must contain a vector of the form $x=\begin{bmatrix}\xi\\ 0\end{bmatrix}$, with $\xi\neq0$.

If (i) holds,  $\Sigma^2$ is not a scalar multiple of the identity along with $\Sigma$ itself, and the invariance of $L$ under the first two and the last transformations from the list \eqref{inva} is equivalent to $L$ being the span of two elements of the standard basis. Conditions imposed on $k_{12},k_{34}$ imply, however, that along with any two standard basis vectors, a subspace invariant under $K$ unavoidably contains the other two, which is a contradiction.

If (ii) holds, then along with $x$, the subspace $L$ will also contain the vectors $\begin{bmatrix}K_1\xi\\ 0\end{bmatrix}, \begin{bmatrix} 0\\ \xi\\\end{bmatrix}, \begin{bmatrix} 0\\ K_2\xi\\\end{bmatrix}$. Condition $\dim L=2$ can hold only if $\xi$ is a common eigenvector of $K_1$ and $K_2$, implying that they commute, which is also a contradiction.
\end{proof} 	

With the case $k(A)=4$ out of the way, we need only to figure out how to distinguish between matrices with $k(A)=2$ and $k(A)=3$.
Indeed, recalling the notion of seeds from the beginning of this section, we know that if $A\in \C^{4\times 4}$ is unitarily irreducible and has at least one seed but does not satisfy conditions of Theorem~\ref{th:k4}, then $k(A)=3$.  {However, it is possible for a matrix $A$ without a seed to satisfy $k(A)=3$, as we will see in Theorem \ref{th:kA3}.}

Following Kippenhahn \cite{Ki,Ki08}, we define a general affine transformation $\tau(z)$ for $z = \xi + i\eta$, $\xi, \eta\in \mathbb{R}$ as

\[\tau(z) = \tau_{a,b,c}(\xi+ i\eta) = a\zeta + ib\eta + c,\]
where $a, b, c\in \mathbb{C}$, $ab\neq 0$ and $ba^{-1}$ is not purely imaginary.

The affine transformation is invertible, with $\tau_{a,b,c}^{-1}  =\tau_{m,n,0}\circ \tau_{0,0,-c},$ where $$m = \frac{b_1}{a_1b_1 - a_2 b_2} -i \frac{a_2}{a_1 b_1 + a_2 b_2}\textmd{ and }n =  \frac{a_1}{a_1 b_1 + a_2 b_2}- i\frac{b_2}{a_1b_1 - a_2 b_2},$$ when $ab\neq 0$ and $ab^{-1}$ is not purely imaginary.

Similarly, we define the affine transformation on a matrix $A = H + iK$ as
\[\tau(A) = \tau_{a,b,c}(A) = aH + ibK + cI_n,\]
for $a, b, c\in \mathbb{C}$.  Kippenhahn proves that $\tau(W(A)) = W(\tau(A))$ \cite[Theorem 4]{Ki,Ki08}, so the affine transformation of the matrix corresponds to the affine transformation on the numerical range. This transformation is also invertible.  Two matrices $A$ and $B$ are called \textit{affinely equivalent} when there exists $\tau$ such that $A = \tau(B)$.  Note that such matrices are unitarily (ir)reducible only simultaneously.  

\begin{lem}\label{lem:afftrans}
Let $A\in \C^{n\times n}$ and assume
there are three orthogonal vectors in $\C^n$ mapping to three distinct points $p_1, p_2, p_3$ on $\partial W(A)$ with distinct corresponding support lines $\ell_1, \ell_2, \ell_3$.  Then $A$ is affinely equivalent to a matrix $A'$, where $W(A')$ is in the first quadrant and is either
	\begin{enumerate}
		\item (parallel case) tangent to each of the lines $x=0$,  $y=0$, and $x=1$ or
		\item (nonparallel case) tangent to each of the lines $x=0$, $y=0$, and $x+y=1$.
	\end{enumerate}
In either case, $k(A)=k(A')$.
\end{lem}

\begin{proof}
At most two of $\ell_1, \ell_2,$ and $\ell_3$ are parallel since $W(A)$ is convex.  Without loss of generality assume that $\ell_1$ and $\ell_2$ are not parallel, and let $m=m_1+m_2i$ denote the intersection point of $\ell_1$ and $\ell_2$.  Then $\ell_1=\{a_1\alpha+a_2\alpha i + m\mid \alpha\in \R\}$ (i.e. the line through $m$ with slope $a_2/a_1$, for some real $a_1,a_2$) and similarly $\ell_2=\{b_1\gamma+b_2\gamma i + m\mid \gamma\in \R\}$.  Let $a=a_2-a_1i$ and $b=b_1+b_2i$, so then $ia=a_1+a_2i$.  Let $\tau$ denote the affine transformation given in Kippenhahn's notation by $\tau_{a,b,m}$, so that for any $\xi+i\eta\in \C$, we have $\tau(\xi+i\eta)=\xi b+i\eta a +m$. We note that Kippenhahn's requirement for invertibility that $ba^{-1}$ is not purely imaginary is guaranteed by the fact that $ia$ and $b$ are not (real) scalar multiples of each other (since $\ell_1$ and $\ell_2$ are not parallel) \cite{Ki,Ki08}.
	
We will now express $\ell_1$ and $\ell_2$ as images of lines under the affine transformation $\tau$.  Specifically, $\ell_2=\{b\gamma+m\mid \gamma\in \R\}=\{\tau(\gamma+0i)\mid \gamma\in \R\}$.  Similarly $\ell_1=\{\tau(0+\alpha i)\mid \alpha\in \R\}$.  Hence $\tau$ maps the $y$-axis to $\ell_1$ and the $x$-axis to $\ell_2$.  Hence $A$ is affinely equivalent to some matrix $B$ for which $W(B)$ is supported by the $x$-axis and the $y$-axis.  By further applying a scaling of $\pm i$ or $\pm 1$, we can guarantee that $W(B)$ lies in the first quadrant and we assume this without loss of generality.
	
We will now shift our focus to $\ell_3$.  In the case that $\ell_3$ is parallel to (without loss of generality) $\ell_1$, we see that $\tau$ maps some line $x=k$ (with $k>0$) to the line $\ell_3$.  Composing our previous $\tau$ with $\tau_{1,k,0}$, we have $\tau\circ \tau_{1,k,0}(\xi + i\eta) = k\xi b + i\eta a + m$	maps $x=1$ to $\ell_3$ while preserving the map of the $y$-axis to $\ell_1$ and the $x$-axis to $\ell_2$.  We conclude that $A$ is affinely equivalent to some matrix $C$ whose numerical range $W(C)$ lies in the first quadrant and has support lines of the $y$-axis, the $x$-axis, and the line $x=1$. This is the parallel case.
	
We now consider the remaining case that the lines $\ell_1,\ell_2,\ell_3$ are pairwise non-parallel.  Let $d$ denote the intersection point of $\ell_3$ and $\ell_2$ (so $d\in \R$) and $e$ the intersection point of $\ell_3$ and $\ell_1$ (so $e$ is purely imaginary).  Then the affine transformation given by $\tau_{e/i,d,0}$ sends the line $\{\gamma+(1-\gamma)i\mid \gamma\in \R\}$ to $\ell_3$.  As in the previous paragraph, $\tau_{e/i,d,0}$ does not affect $\ell_1$ and $\ell_2$.  We conclude that $A$ is affinely equivalent to some matrix $C$ whose numerical range $W(C)$ lies in the first quadrant and has the $x$-axis, the $y$-axis, and the line $x+y=1$ as support lines.
	
In both cases, $k(A)$ is preserved by the affine transformations since there exists some orthonormal set $\{{y}_1, {y}_2, {y}_3\}\in \C^n$ such that ${y}_1^*C{y}_1$ is on the $x$-axis, ${y}_2^*C{y}_2$ is on the $y$-axis, and ${y}_3^*C{y}_3$ is on the line $x=1$ or $x+y = 1$, depending on the case \cite[Proof of Theorem 4]{Ki,Ki08}.  Therefore $C$ is the desired matrix $A'$.
\end{proof}

\begin{thm}\label{th:kA3}
A $4\times4$ matrix $A$ is unitarily irreducible and has $k(A)=3$ arising from three points not on seeds if and only if, up to affine equivalency, it is unitarily similar to $H+iK$ of the following structure.
Either \\ \noindent {\rm (a)}
\begin{equation}\label{parallel}
H = \left[\begin{array}{rrrr}
0 & 0 & 0 & 0\\
0 & h_{22} & 0 & h_{24}\\
0 & 0 & 1 & 0\\
0 & \overline{h_{24}}  & 0 & h_{44} \end{array}\right], \quad
K =  \left[\begin{array}{rrrr}
k_{11} & 0 & k_{13}  & k_{14}\\
0 & 0 & 0 & 0\\
\overline{k_{13}} & 0 & k_{33} & k_{34} \\
\overline{k_{14}} &  0  &\overline{k_{34}}& k_{44}
\end{array}\right],
\end{equation}
\eq{par1} 0<\begin{bmatrix}h_{22} & h_{24}\\ \overline{h_{24}} & h_{44}\end{bmatrix}<I,\quad
0<\begin{bmatrix} k_{11} & k_{13} & k_{14}\\ \overline{k_{13}}  & k_{33} & k_{34} \\
	\overline{k_{14}}  &\overline{k_{34}}& k_{44}\end{bmatrix}, \en
\eq{par2} h_{24}\neq 0 \text{ and at most one of } k_{13}, k_{14}, k_{34} \text{ is equal to zero}, \en

or  \\ \noindent {\rm (b)}

\begin{equation}\label{nonparallel}
H = \left[\begin{array}{rrrr}
0 & 0 & 0 & 0\\
0 & h_{22} & 0 & h_{24}\\
0 & 0 & h_{33} & h_{34}\\
0 & \overline{h_{24}} & \overline{h_{34}} & h_{44}
\end{array}\right], \quad
K =  \left[\begin{array}{rrrr}
k_{11} & 0 & 0 & k_{14}\\
0 & 0 & 0 & 0\\
0 & 0 & k_{33} & -h_{34}\\
\overline{k_{14}} &  0  &-\overline{h_{34}}& k_{44}
\end{array}\right],
\end{equation}
\eq{nonpar1} 0<\begin{bmatrix} h_{22} & 0 & h_{24}\\ 0 & h_{33} & h_{34} \\
\overline{h_{24}} & \overline{h_{34}} & h_{44}\end{bmatrix},\,  0<\begin{bmatrix} k_{11} & 0 & k_{14}\\ 0 & k_{33} & -h_{34} \\
\overline{k_{14}} & -\overline{h_{34}} & k_{44}\end{bmatrix},
 \begin{bmatrix} k_{11} & 0 & k_{14}\\ 0 & h_{22} & h_{24} \\
	\overline{k_{14}} & \overline{h_{24}} & h_{44}+k_{44}\end{bmatrix}<(h_{33}+k_{33})I,\en
and \eq{nonpar2} h_{24}, h_{34}, k_{14}\neq 0. \en
\end{thm}

\begin{proof}
{\sl Necessity.}  Conditions imposed on $A$ imply that there are three orthonormal vectors $f_j\in\C^4$ whose images under $f_A$ lie on three distinct supporting lines $\ell_j$ ($j=1,2,3$, respectively) of $W(A)$. There are two possibilities, depending on whether or not two of them are parallel. {We complete} the set $\{f_j\}_1^3$ to create an orthonormal basis $\mathcal F$ of $\C^4$, and consider these two possibilities separately.

{\sl Case 1.} Two of the lines, say $\ell_1$ and $\ell_3$, are parallel. According to Lemma~\ref{lem:afftrans} we may without loss of generality suppose that $W(A)$ lies in the first quadrant, $\ell_1$ and $\ell_2$ are the $x$- and $y$-axes, respectively, and $\ell_3$ is the vertical line with the abscissa $x=1$. It is clear that in the basis $\mathcal F$, the matrix of $A$ (after the affine transformation) takes the form \eqref{parallel}. Moreover, the absence of seeds {guarantees that the relevant eigenvalues are simple, hence the strict inequalities in \eqref{par1}}. In its turn, \eqref{par2} is required for unitary irreducibility. Indeed, $f_1, f_2$ or $f_3$ is a common eigenvector of $H,K$ from \eqref{parallel} if, respectively, $k_{13}=k_{14}=0$,  $h_{24}=0$, or $k_{13}=k_{34}=0$. Finally, if $k_{14}=k_{34}=0$, then $\Span\{f_2,f_4\}$ is invariant under both $H$ and $K$. So, (a) holds.

{\sl Case 2.} No parallel pairs of lines among $\ell_j$, $j=1,2,3$. Another application of Lemma~\ref{lem:afftrans} allows us to suppose without loss of generality that $W(A)$ lies in the first quadrant, $\ell_1$ is the $x$-axis, $\ell_2$ is the $y$-axis, and $\ell_3$ intersects the two at the same distance from the origin. The matrix of $A$ in the basis $\mathcal F$ after the respective affine transformation then satisfies \eqref{nonparallel}; note that $h_{34}+k_{34}=0$ comes from $H+K$ being proportional to the Hermitian part of $Ae^{i\pi/4}$.
Conditions \eqref{nonpar1} are equivalent to $\ell_j$ containing no seeds, and in absence of \eqref{nonpar2} $H$ and $K$ would have common eigenvectors. So, (b) holds.

{\sl Sufficiency.}  Given \eqref{parallel}--\eqref{par1} or \eqref{nonparallel}--\eqref{nonpar1}, the points $f_A(e_j)$ (with $\{e_j\}$ denoting the standard basis vectors) lie on the boundary of $W(A)$.
Consequently, $k(A)\geq 3$.

To show that $A$ is unitarily irreducible under (a), observe that $e_2$ is an eigenvector of $K$ corresponding to its simple eigenvalue, and therefore for any reducing subspace $L$ of $A$,  either $e_2\in L$ or $e_2\in L^\perp$. Without loss of generality, suppose the former. Then also $He_2\in L$, implying that $e_4\in L$ since $h_{24}\neq 0$. Next, $Ke_4\in L$, implying that $k_{14}e_1{+k_{34}}e_3\in L$. If $k_{14}=0$, then ${k_{34}}\neq 0$ due to \eqref{par2}, and so $e_3\notin L^\perp$. Being an eigenvector of $H$ corresponding to its simple eigenvalue, $e_3$ then has to lie in $\mathcal L$. From $Ke_3\in L$, we then conclude that $e_1\in L$ (since $k_{13}\neq 0$ due to \eqref{par2}), and $ L=\C^4$. Similarly, if $k_{14}\neq 0$, then $e_1\notin L^\perp$, implying $e_1\in L$, and thus ${k_{34}}e_3\in L$. Also, $Ke_1\in L$ along with $e_1$, and so $k_{13}e_3\in L$. Since at least one of the coefficients $k_{13},{k_{34}}$ is non-zero, we conclude that $e_3\in L$, again implying that $ L=\C^4$.

Under (b), we are dealing with arrowhead matrices. By taking a linear combination of $H$ and $K$ such that the three diagonal entries are distinct, any eigenvector orthogonal to $e_4$ must be one of $e_1,e_2,e_3$; conditions \eqref{nonpar2} thus guarantee there are no common eigenvectors of $H$ and $K$ orthogonal to $e_4$. At the same time, for any common invariant subspace $ L$ of $H$ and $K$, since $e_1$ is a simple eigenvector of $H$ it must be contained in either $ L$ itself or its orthogonal complement $ L^\perp$.  By Lemma~\ref{l:red}, $ L=\C^4$ or $ L^\perp=\C^4$. Hence, $A$ is unitarily irreducible.

Moreover, still under (b), for any $t\in\R$ \[ K+tH=\begin{bmatrix}k_{11} & 0 & 0 & k_{14}
\\ 0 & th_{22} & 0 & th_{24} \\ 0 & 0 & k_{33}+th_{33} & {(t-1)}h_{34} \\
\overline{k_{14}} & t\overline{h_{24}} & {(t-1)}\overline{h_{34}} & k_{44}+th_{44}
 \end{bmatrix}  \]
is an arrowhead Hermitian matrix. Its extremal eigenvalues are therefore simple for all values $t\neq 0,1$ due to
\eqref{nonpar2}.
In particular, $k(A)$ cannot equal 4, and thus it is indeed 3.

Returning to setting (a), we have
\eq{kth} K+tH=\begin{bmatrix} k_{11} & 0 & k_{13} & k_{14}
\\ 0 & t h_{22} & 0 & th_{24} \\ \overline{k_{13}} & 0 & k_{33}+t & k_{34} \\
\overline{k_{14}} & t\overline{h_{24}} & \overline{k_{34}} & k_{44}+t h_{44}\end{bmatrix}.  \en
This matrix cannot have only two distinct eigenvalues. Indeed,
if this were the case, both of them would be the extremal eigenvalues of \eqref{kth} and thus differ from $th_{22}$, as guaranteed by the condition $h_{24}\neq 0$ from \eqref{par2}. The remaining part of \eqref{par2} implies then that the left lower $3\times3$ minor of $K+tH-\lambda I$ vanishes for at most one value of $\lambda$. Therefore, having two multiple eigenvalues of multiplicity two is impossible.

On the other hand, a triple eigenvalue also is not an option, because the rank of $K+tH-\lambda I$ never falls below two.
\end{proof}

Note that in case (b) we actually proved a stronger claim: since there are no seeds, the only $z_j=\scal{Af_j,f_j}\in\partial W(A)$ generated by orthonormal triples $\{f_j\}_{j=1}^3$ lie on three distinct supporting lines of $W(A)$. In case (a), however, seeds are not ruled out. An obvious way in which they can occur is when the matrix $K$ in \eqref{parallel} has a multiple maximal eigenvalue. Other cases, if any, materialize when the matrix \eqref{kth} has a multiple extremal eigenvalue $\lambda$. Further analysis shows that this happens for a unique choice of
\eq{tl} t= k_{11}-k_{33}+\frac{k_{13}k_{34}}{k_{14}}-\frac{\overline{k_{13}}k_{14}}{k_{34}}, \quad \lambda=k_{11}-\frac{\overline{k_{13}}k_{14}}{k_{34}} \en
if and only if $k_{13}\overline{k_{14}}k_{34}\in\R\setminus\{0\}$, $\frac{k_{13}k_{34}}{k_{14}}+\frac{\overline{k_{13}}k_{14}}{k_{34}}$ and
$th_{22}-\lambda$ are of the same sign, and
\eq{45}
\left(\frac{k_{13}k_{34}}{k_{14}}+\frac{\overline{k_{13}}k_{14}}{k_{34}}\right)\det
\begin{bmatrix}th_{22}-\lambda & th_{24} \\ t\overline{h_{24}} & k_{44}+th_{44}-\lambda \end{bmatrix}
=(th_{22}-\lambda)\left(\abs{k_{14}}^2+\abs{k_{34}}^2\right).
\en
In particular, there are no seeds in case (a) if (exactly) one of $k_{13},k_{14},k_{34}$ is equal to zero. We demonstrate the possibility of seeds in case (a) in the following example.

\begin{figure}[h]
\includegraphics[width=1in]{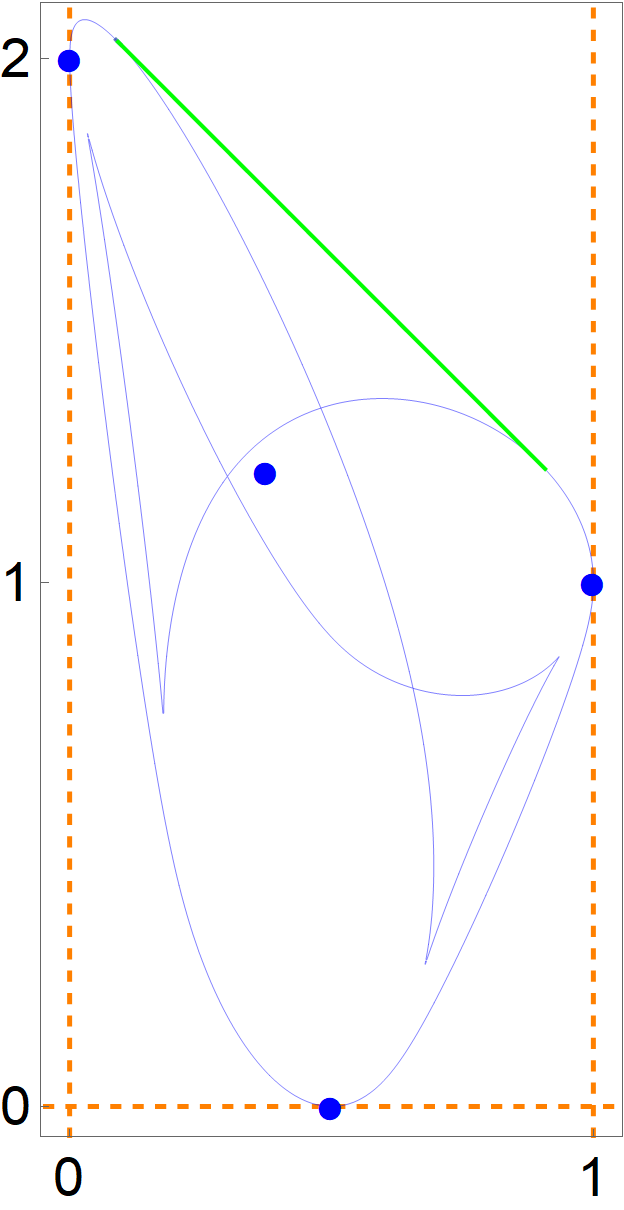}
\caption{The boundary generating curve and seed of the matrix in Equation \eqref{k3mat}.
}
\label{fig:a1}
\end{figure}

\begin{ex}
The matrix \eq{k3mat}\left[
\begin{array}{cccc}
2 i & 0 & \frac{i}{8} & \frac{i}{4} \\
0 & \frac{1}{2} & 0 & -\frac{1}{4} \\
\frac{i}{8} & 0 & 1+i & -\frac{i}{4} \\
\frac{i}{4} & -\frac{1}{4} & -\frac{i}{4} & \frac{3}{8}+\frac{63 i}{52} \\
\end{array}
\right]\en is of the form found in case (a) of Theorem \ref{th:kA3}.  This matrix satisfies \eqref{tl} and \eqref{45} with $t=1$ and $\lambda=17/8$.  Its boundary generating curve is in Figure \ref{fig:a1}, where the dotted lines are the support lines $x=0, y=0, x=1$ and the marked points are the images of $e_1,e_2,e_3,e_4$.
 Observe that the numerical range has a seed in the form of a flat portion, shown by the line segment at the top, which gives rise to a second set of three orthonormal vectors mapping to the boundary of the numerical range.
\end{ex}

We conclude this paper with an example to illustrate that the Gau--Wu number is not stable with respect to the entries of the matrix.

\begin{ex}Let $A$ denote a unitarily irreducible matrix of the form \eqref{ah}, satisfying $b_j=\overline{c_j}$ for all $j$ and satisfying the conditions of Theorem \ref{th:kA3}. Then $k(A)=3$.  By Theorem \ref{th:arrow}, a small perturbation to change some of these entries (say by adding some small $\epsilon$ to $b_j$, for each $j\in J_0$) will result in a matrix with Gau--Wu number 2.
\end{ex}

\bibliographystyle{amsplain}
\bibliography{master}

\end{document}